\documentclass[12pt]{amsart}

\usepackage{amsthm}
\usepackage{amscd}
\usepackage{amsfonts}
\usepackage{amsmath}
\usepackage{amssymb}
\usepackage{amsrefs}
\usepackage{mathrsfs}
\usepackage{enumitem}
\usepackage{multirow}
\usepackage{verbatim}
\usepackage{url}
\usepackage{graphicx}

\usepackage{anyfontsize}
\usepackage{t1enc}

\usepackage{microtype}
\usepackage{hyperref}

\usepackage{color}

\definecolor{brightpink}{rgb}{1.0, 0.0, 0.5}

\usepackage{hyperref}

\date{August 2023}

\theoremstyle{plain}
\newtheorem{theorem}{Theorem}[section]
\newtheorem{corollary}[theorem]{Corollary}

\newtheorem{question}[theorem]{Question}

\theoremstyle{definition}

\newtheorem{definition}[theorem]{Definition}

\DeclareMathOperator{\R}{\mathbb{R}}
\DeclareMathOperator{\C}{\mathbb{C}}

\DeclareMathOperator{\N}{\mathbb{N}}

\author{Ilijas Farah}\address[Farah]
{Department of Mathematics and Statistics\\ York  University\\ To\-ron\-to, Ontario M3J 1P3\\ CANADA}
\email{ifarah@yorku.ca}

\author{Jeffrey Marshall-Milne}\address[Marshall-Milne]
{Department of Mathematics and Statistics\\ McMaster University\\ Ha\-mil\-ton, Ontario L8S 4L8\\ CANADA}
\email{marshj16@mcmaster.ca}

\thanks{Partially supported by NSERC}

\title[Chain Conditions and the Axiom of Choice]{Chain Conditions and the Axiom of Choice}


\begin{document}

\begin{abstract}
Within the framework of Zermelo-Fraenkel set theory without the Axiom of Choice, we establish equivalents to the assertion "the union of a countable collection of finite sets is countable" in the context of metric spaces, probability theory, and Hilbert spaces. The categorical principle underlying these equivalences is identified.
\end{abstract} 

\maketitle

\section{Introduction}

Adding to the large body of literature on choice principles and their equivalents (see \cite{A2} and \cite{A4}), we work in Zermelo-Fraenkel set theory without the Axiom of Choice (abbreviated ZF). Modern treatments of real analysis make heavy use of the Axiom of Choice (AC). As a result, we frequently encounter theorems which may themselves be viewed as choice principles. A typical example is the equivalence of the Hahn-Banach Extension Theorem and the Boolean Prime Ideal Theorem. The program of this paper is to establish more such theorem/choice-principle pairs, typically for fairly weak choice principles.

As we will see, these arise very naturally when studying chain conditions. A topological space is said to satisfy the \textit{countable chain condition} (or it is simply called ccc) if there does not exist an uncountable collection of pairwise disjoint open sets. Given a cardinal~$\kappa$, there is also the concept of a $\kappa$-cc topological space, where every collection of pairwise disjoint open sets has cardinality $<\kappa$. A Boolean algebra is said to be ccc if the associated Stone space is ccc. A standard ZF+AC argument shows that a Boolean algebra which admits a probability measure satisfies a rather strong form of the countable chain condition, called the $\sigma$-bounded chain condition.

The utility of the countable chain condition (see e.g., \cite{A5}) dates back at least to P.J. Cohen's proof of the independence of the Continuum Hypothesis from ZF+AC. It is commonly used to prove that a forcing notion does not collapse cardinals or change cofinalities when one passes from a model of ZF+AC to its forcing extension. When AC is at our disposal, the countable chain condition admits many useful equivalents in different contexts. For example, the countable chain condition, separability, and second countability are all equivalent conditions for metric spaces. The situation becomes more complicated when we omit the Axiom of Choice (see \cite{A6}). 

Serving as a jumping-off point for this project, D.H. Fremlin's Exercise 561Yc \cite{A3} shows that these equivalences need not hold when AC fails. In part iii of this exercise, Fremlin notes that the existence of so called "Russell socks" implies the existence of a complete and totally bounded metric space which is non-compact and does not satisfy the countable chain condition (and is therefore neither separable nor second-countable). This observation is made precise in \ref{t30}, and is generalized in \ref{t7}. In part iv of the same exercise, Fremlin notes that the existence of Russell socks further implies the existence of a probability algebra which is not ccc. \ref{t16} makes this implication precise, whereas \ref{t15} establishes a similar observation in the context of Hilbert spaces. The categorical principle unifying \ref{t30}, \ref{t16}, and \ref{t15} is given by \ref{z1}.

This paper is inspired by \cite{i1}, \cite{i3}, and \cite{i2}. Of note is the closely related work found in \cite{c3}, \cite{c4}, \cite{c2}, and \cite{c1}. \cite{d1} and \cite{A2} were indispensable resources in the writing of this paper.

\subsection*{Acknowledgment}
 We would like to thank Asaf Karagila for valuable comments on an early version of this paper, as well as his words of encouragement. This paper would not have found its current form without his support.

\section{Metric Spaces}

Let $\kappa$ be a cardinal number (not necessarily well-ordered or infinite).
\begin{definition}
    A topological space is said to satisfy the \textit{$\kappa^\leq$-chain condition} if every collection of pairwise disjoint open sets has cardinality~$\leq \kappa$.
\end{definition}

Note that the case $\kappa = \aleph_0$ corresponds to the countable chain condition. In presence of the Axiom of Choice, the $\kappa^\leq$-chain condition reduces to the $\kappa^+$-chain condition, where $\kappa^+$ is the least cardinal greater than $\kappa$. However, in section 3 we will consider the $\kappa^\leq$ chain condition for cardinals which are not necessarily well-ordered $\kappa$.  Naturally, a metric space is called $\kappa$-cc (or $\kappa^\leq$-cc) if it is $\kappa$-cc (respectively, $\kappa^\leq$-cc) as a topological space. Assertion ii of \ref{t7} is a standard theorem of ZF+AC. \ref{t7} makes its relationship to the Axiom of Choice precise.

\begin{theorem}\label{t7}
Let $\kappa$ be an infinite well-ordered cardinal. The following are equivalent.
\begin{itemize}
    \item[\textnormal{i}] If $X=(X_1, X_2, ...)$ is a sequence of sets such that $|X_n|< \kappa$ for all $n\in \N$, then $|\bigcup_n X_n|\leq \kappa$. 
    \item[\textnormal{ii}] Suppose $Y$ is a metric space such that for every $\epsilon > 0$, there exists a set of points $F \subseteq Y$ satisfying $|F| < \kappa$ and $Y = \cup_{x\in F}B(x, \epsilon)$ (where $B(x, \epsilon):= \{y\in Y : d(x, y)<\epsilon\}$). Then $Y$ satisfies the $\kappa^\leq$-chain condition.
\end{itemize}
\end{theorem}

If we further assume that $\kappa$ has uncountable cofinality (i.e. every unbounded subset is uncountable), we get the following.

\begin{theorem}\label{69} If $\kappa$  is a well-ordered cardinal with  uncountable cofinality then the following are equivalent.
\begin{itemize}
    \item[\textnormal{i}] If $X=(X_1, X_2, ...)$ is a sequence of sets such that $|X_n|< \kappa$ for all $n\in \N$, then $|\bigcup_n X_n|< \kappa$. 
    \item[\textnormal{ii}] Suppose $Y$ is a metric space such that for every $\epsilon > 0$, there exist a set of points $F \subseteq Y$ such that $|F| < \kappa$ and $Y = \cup_{x\in F}B(x, \epsilon)$. Then $Y$ satisfies the $\kappa$-chain condition.
\end{itemize}
\end{theorem}

The proofs of these two theorems are sufficiently similar to be presented simultaneously. 

\begin{proof}[Proof of \ref{t7} and \ref{69}] ($\Rightarrow$) Assume i holds. Let $Y$ be a metric space satisfying the hypothesis of~ii. Let~$\mathcal{C}$ be a collection of pairwise disjoint non-empty open subsets of $Y$, and let
\begin{align*}
    \mathcal{C}_n = \{ U\in \mathcal{C} : \exists x\in U \text{ such that } B(x, 1/n) \subseteq U \}. 
\end{align*}
Clearly $\mathcal{C}=\bigcup_n \mathcal{C}_n$. To see that $|\mathcal{C}_n|< \kappa$, let $F\subseteq Y$ be such that $|F| < \kappa$ and $Y = \cup_{y\in F}B(y, 1/n)$. Endow $F$ with a well-order. By the construction of $F$, the set $F \cap U$ is non-empty for all $U \in \mathcal{C}_n$. Let $j:\mathcal{C}_n \rightarrow F$ be the map defined by letting $j(U)\in F$ be the minimal element satisfying $j(U)\in U$. Then $j$ is an injection because the members of $\mathcal{C}_n$ are pairwise disjoint. So $|\mathcal{C}_n|< \kappa$. By i, $|\mathcal{C}|=|\bigcup_n\mathcal{C}_n|\leq \kappa$ (if $\kappa$ has uncountable cofinality, this becomes $|\mathcal{C}|<\kappa$), as desired.

($\Leftarrow$) Assume ii holds. Let $X = (X_1, X_2, ...)$ be as in the hypothesis of i. Assume without loss of generality that the $X_i$'s are pairwise disjoint, and that $0 \notin \bigcup_n X_n$. Let $Y_n= X_n \cup \{0\}$ with the discrete metric, and let $Y = \prod_{n\in \N} Y_n$ with the product metric $d(\overline{x}, \overline{y})=\sum_{i=1}^\infty 2^{-i}d_i(x_i, y_i)$. To see that $Y$ satisfies the hypothesis of ii, let $\epsilon > 0$ and choose $n\in \N$ such that $2^{-n} < \epsilon$. Let $F = \{\overline{x} \in Y : \forall i>n, x_i=0\}$. Then the set of open $\epsilon$-balls with centers in $F$ covers $X$. Furthermore, $|F|=|Y_1 \times ... \times Y_n|$. But $Y_i < \kappa$ for all $i \in \N$, and $\kappa$ is an infinite well-ordered cardinal. Therefore, it follows from basic set-theoretic considerations (namely the fact that for infinite well-ordered cardinals $\xi$ and $\eta$ we have $\xi+\eta=\xi\eta=\max(\xi,\eta)$)  that $|Y_1 \times ... \times Y_n| < \kappa$. Ergo, $Y$ satisfies the hypothesis of ii.

To see that $|\bigcup_n X_n|\leq \kappa$ (or $|\bigcup_n X_n|< \kappa$ in the case that $\kappa$ has uncountable cofinality), for $x\in X_n$, let $U_x = \{0\}^{n-1}\times \{x\}\times \prod_{i > n}Y_i.$ Then $\mathcal{C}=\{U_x : x\in \bigcup_n X_n\}$ is a collection of pairwise disjoint open sets, so $|\mathcal{C}|\leq \kappa$ by ii (or $|\mathcal{C}|< \kappa$ if $\kappa$ has uncountable cofinality). But $\mathcal{C}$ is in one-to-one correspondence with $\bigcup_n X_n$ via $U_x \leftrightarrow x$.
\end{proof}

Recall that a metric space $X$ is said to be \textit{totally bounded} if for every $\epsilon > 0$, $X$ can be covered by finitely many open $\epsilon$-balls. Taking $\kappa = \aleph_0$, Assertion ii of \ref{t7} takes the form of a well-known ZF+AC theorem on totally bounded metric spaces. The following corollary (much in the spirit of P. Howard's Proposition 2.7 \cite{B1}) shows how this theorem may be viewed as a choice principle.

\begin{corollary}\label{t30}
The following are equivalent.
\begin{itemize}
    \item[\textnormal{i}] The union of a countable collection of finite sets is countable.
    \item[\textnormal{ii}] Every totally bounded metric space is ccc.
\end{itemize}
\end{corollary}

\ref{t7}, \ref{69}, and \ref{t30} will serve as our prototypical examples in Section 3 as we generalize their proof to every class of partial orders satisfying certain conditions.

\section{The Categorical Perspective}

Given a partially ordered set (aka a poset) $P$, we say two elements $x, y \in P$ are \textit{incompatible} if they do not share a common lower bound (aside from a lower  bound for all of $P$, if it exists). An \emph{antichain} is a subset $\mathcal{C} \subseteq P$ consisting of pairwise incompatible elements. Given a cardinal $\kappa$, we can generalize the chain conditions defined for topological spaces to any partial order by stating that $P$ is $\kappa$-cc if every collection of pairwise incompatible elements has cardinality $<\kappa$. The $\kappa^\leq$-chain condition is similarly generalized, with the $\aleph_0^\leq$-chain condition again being referred to interchangeably as the countable chain condition. 

The following is a generalization of a notion introduced in \cite{horn1948measures} that will be discussed further at the end of the present section.  

\begin{definition}
    A poset $P$ is said to satisfy the \textit{$\sigma^\kappa$-chain condition} if there exists a countable collection of $\kappa$-cc subsets $P_1, P_2, ... \subseteq P$ such that $P=\bigcup_n P_n$. The case $\kappa = \aleph_0$ is called the $\sigma$-finite chain condition.
\end{definition}

We will say that  a cardinal $\kappa$ (not necessarily well-ordered) has \emph{uncountable cofinality} if a set of cardinality $\kappa$ cannot be covered by a countable family of sets of cardinality $<\kappa$ each. (This definition clearly does not depend on the choice of set of cardinality $\kappa$.) 
 If $\kappa$ is well-ordered, this condition is equivalent to every unbounded subset of~$\kappa$ being uncountable.

The starting point for Theorem~\ref{z1} below is the following simple observation. If the union of every family of $\aleph_0$ sets of cardinality $<\kappa$ each has cardinality $\leq\kappa$, then every poset satisfying the $\sigma^\kappa$ chain condition must also satisfy the $\kappa^\leq$ chain condition. If, in addition, $\kappa$ has uncountable cofinality, then every poset satisfying the $\sigma^\kappa$ chain condition must also satisfy the $\kappa$ chain condition.

\begin{theorem}\label{z1}
Let $\kappa$ be a cardinal (not necessarily well-ordered or infinite), and let $\mathfrak{A}_\kappa$ denote the proper class of all sequences of pairwise disjoint non-empty sets $X=(X_1, X_2, ...)$ such that $|X_n| < \kappa$ for all $n\in \N$.
Suppose that $\mathfrak{B}$ is a family of posets which satisfy the $\sigma^\kappa$ chain condition. Further suppose that for every $X\in \mathfrak{A}_\kappa$, there is some $\mathcal F(X)\in \mathfrak{B}$ and an injection $f\colon \bigcup_n X_n\to \mathcal F(X)$ whose range is an antichain. 

Then every poset in $\mathfrak{B}$ is $\kappa^\leq$-cc if and only if given any countable collection of sets, each of cardinality $<\kappa$, their union has cardinality $\leq \kappa$.

If, in addition, $\kappa$ has uncountable cofinality, then every poset in $\mathfrak{B}$ is $\kappa$-cc if and only if given any countable collection of sets, each of cardinality $<\kappa$, their union has cardinality $< \kappa$.
\end{theorem}

 To shed some light on Theorem \ref{z1}, note that Corollary~\ref{t30} corresponds to the special case where $\kappa = \aleph_0$ and $\mathfrak{B}$ is the proper class of all topologies (denoted $\tau(X)$) of totally bounded metric spaces~$X$, each partially ordered by set inclusion. In this case, $\mathcal{F}$ maps a sequence $X=(X_1, X_2, ...)$ to $\tau(\prod_n (X_n \cup \{X_n\}))$, and $f$  maps $x\in X_n$ to $(\prod_{i<n} \{X_n\}) \times \{x\} \times ( \prod_{i > n} (X_n \cup \{X_n\}))$ in  $\tau(\prod_n (X_n \cup \{X_n\})) = \mathcal{F}(X)$.

\begin{proof}[Proof of Theorem~\ref{z1}]  
		To prove the nontrivial implication, note that if $Y=(Y_1,Y_2,\dots)$ is a sequence of sets such that $|Y_n|<\kappa$ for all $n$, then $X_n=Y_n\setminus 	\bigcup_{j<n} Y_j$ defines an element of $\mathfrak{A}_\kappa$ such that $\bigcup_n X_n=\bigcup_n Y_n$. 
		  It therefore suffices to consider elements of $\mathfrak{A}_\kappa$. Let $X=(X_1, X_2, ...) \in \mathfrak{A}_\kappa$ and let $f:\bigcup_n X_n \rightarrow \mathcal{F}(X)$ be an injection whose image is an antichain. Since $\mathcal F(X)$ is $\kappa^\leq$-cc, we have $|f(\bigcup_n X_n)| \leq \kappa$. But $f(\bigcup_n X_n)$ is in bijection with $\bigcup_n X_n$, so $|\bigcup_n X_n|\leq \kappa$, as desired. 
		  
		  If $\kappa$ has uncountable cofinality, then for every $X\in \mathfrak{A}_\kappa$ we have $|\bigcup_n X_n|\leq \kappa$ if and only if $|\bigcup_n X_n|<\kappa$, and the desired equivalence reduces to the one proven for an arbitrary $\kappa$.  
		  \end{proof}

A typical example of a poset satisfying the $\sigma$-finite chain condition is the set of nonzero elements of a probability measure algebra. Setting~$P_n$ to be the family of all sets of measure $>1/n$, for $n\geq 2$, it is clear that~$P_n$ cannot contain an antichain of cardinality $n$. Since the cardinality of antichains in $P_n$ is uniformly bounded, we have an even stronger chain condition, called the $\sigma$-bounded chain condition. These two chain conditions were introduced by Horn and Tarski in the course of their study of measure algebras \cite{horn1948measures}. It took until \cite{thummel2014problem} to prove that the $\sigma$-bounded chain condition is stronger than the $\sigma$-finite chain condition. While the $\sigma$-finite chain condition implies the countable chain condition in the presence of the Axiom of Choice, we have the following corollary of Theorem~\ref{z1}, obtained by setting $\kappa=\aleph_0$. 

\begin{corollary}
The following are equivalent.
\begin{itemize}
	\item[\textnormal{i}] The union of a countable collection of finite sets is countable.
	\item[\textnormal{ii}] Every poset satisfying the $\sigma$-finite chain condition  is ccc. 
	\item[\textnormal{iii}] Every poset satisfying the $\sigma$-bounded chain condition is ccc.
\end{itemize}
\end{corollary}

The rest of this paper is focused on deriving corollaries of \ref{z1} in the context of probability theory and Hilbert spaces.

\section{Probability}

Given a Boolean algebra $\textbf{B}=(B, \textbf{0}, \textbf{1}, \wedge, \vee)$ and elements $x, y \in B$, we write $x \leq y$ if $x \wedge y = x$. In this way, $B$ admits a canonical poset structure. $\textbf{B}$ is said to be ccc if the poset $B$ is ccc. Note that elements $x, y \in B$ are incompatible in the sense of posets if and only if $x \wedge y = \textbf{0}$. 

A \emph{probability measure algebra} is a pair $(\textbf{B}, \mu)$, where $\textbf{B}=(B, \textbf{0}, \textbf{1}, \wedge, \vee)$ is a Boolean algebra and $\mu: B \rightarrow [0, 1]$ is a finitely additive probability measure (i.e. $\mu(\textbf{1})=1$) assigning a positive measure to all non-zero elements. Note that we do not require $\textbf{B}$ to be complete. (Recall that a Boolean algebra is \emph{complete} if every nonempty ßsubset has a supremum.)    Given a probability algebra $(\textbf{B}=(B, \textbf{0}, \textbf{1}, \wedge, \vee), \mu)$, we define $s(\textbf{B}, \mu)=B$; this is a forgetful functor from the category of all probability algebras into the category of Boolean algebras.  As previously mentioned, when AC is at our disposal, every probability algebra is ccc. The following theorem makes the role of AC much more precise.

\begin{theorem}\label{t16} The following are equivalent.
\begin{itemize}
    \item[\textnormal{i}] The union of a countable collection of finite sets is countable.
    \item[\textnormal{ii}] Every probability algebra is ccc.
\end{itemize}
\end{theorem}

\begin{proof}
We employ the framework of \ref{z1} with $\kappa=\aleph_0$, and where $\mathfrak{B}$ is the proper class of all posets $s(\textbf{B}, \mu)$ (for $(\textbf{B}, \mu)$ ranging over all probability algebras).

($\Rightarrow$) Let $B=s(\textbf{B}, \mu)\in \mathfrak{B}$ and let $\mathcal{C}\subseteq B$ be an antichain. Without loss of generality, assume $\textbf{0} \notin \mathcal{C}$. Let
\begin{align*}
    \mathcal{C}_n = \bigg\{ x \in \mathcal{C} : \mu(x) > \frac{1}{n} \bigg\}
\end{align*}
so that $\mathcal{C}=\bigcup_n \mathcal{C}_n$. By additivity of $\mu$, we have $|\mathcal{C}_n| \leq n$ for all $n$, so $\mathcal{C}$ is a union of countably many finite sets. Thus i implies ii, as desired.

($\Leftarrow$) For this direction, we need to construct the map $\mathcal{F}: \mathfrak{A}_{\aleph_0} \rightarrow \mathfrak{B}$ required by \ref{z1}. Given a sequence of pairwise disjoint non-empty finite sets $X=(X_1, X_2, ...)\in \mathfrak{A}_{\aleph_0}$, define $\textbf{B}_X$ to be the Boolean algebra of all subsets of $\bigcup_n X_n$. For $a\in X_n$, define $\mu_X(\{a\})=|X_n|^{-1}2^{-n}$, and for an arbitrary subset $S \subseteq \bigcup_n X_n$, define
\begin{align*}
    \mu_X(S) = \sup_{F \subseteq S \text{ finite}}\sum_{a\in F}\mu_X(\{a\}).
\end{align*}
Then we define $\mathcal{F}$ by $\mathcal{F}(X) = s(\textbf{B}_X, \mu_X)$. To satisfy the hypothesis of Theorem~\ref{z1}, define $f: \bigcup_n X_n \rightarrow \mathcal{F}(X)$ by mapping $a\in \bigcup_n X_n$ to the singleton $\{a\} \in \mathcal{F}(X)$. By \ref{z1}, we conclude that ii implies i, as desired.
\end{proof}

\section{Hilbert Spaces}

In the presence of Countable Choice (that is, the assertion that every countable collection of nonempty sets admits a choice function, denoted $\text{AC}_\omega$), a metric space is Cauchy complete (i.e. every Cauchy sequence converges) if and only if every nested sequence of non-empty closed sets with vanishing diameter intersects in a singleton. Ergo, when working with $\text{AC}_\omega$, it is customary to generically call a metric space satisfying one of these two conditions "complete." However, in the absence of any form of choice, Cauchy completeness becomes a weaker condition. 
We follow  \cite{A1} and say that a metric space in which every nested sequence of non-empty closed sets with vanishing diameter intersects in a singleton is \emph{$\sigma$-complete} (some authors use the term \emph{Cantor complete}). As demonstrated in Section 2 of \cite{A1}, an inner product space which is $\sigma$-complete retains the fundamental properties of a Hilbert space such as the closest vector property, the existence of orthogonal complements, and the Riesz representation theorem. We follow \cite{A1} (Definition 1.0.1.ii) and say that an inner product space is a \emph{Hilbert space} if it is  $\sigma$-complete with respect to the induced norm.

It is straightforward to verify that, the theory of orthogonal projections being unchanged in the absence of choice, we can still obtain Bessel's inequality. Hilbert spaces and operators on Hilbert spaces in the absence of AC have been studied in \cite{BrunnerLinear} and \cite{BrunnerSB} from a slightly different point of view; see the discussion in 6.1.2 of \cite{A1}. 
%
%

Given a set $X$, we define $\ell^2(X)$ to be the set of all maps $f:X \rightarrow \C$ such that $\sum_{x\in X}|f(x)|^2<\infty$ (\cite{A1}, Definition 3.0.2). By equipping $\ell^2(X)$ with the obvious inner product (\cite{A1}, 3.0.3), one shows that $\ell^2(X)$ is a Hilbert space (\cite{A1}, Proposition 3.0.4) and that the collection of characteristic maps $\zeta_x: X \rightarrow \C$ of singletons $\{x\}\subseteq X$ forms an orthonormal basis (\cite{A1}, 3.0.5). In fact, Hilbert spaces constructed this way are the canonical Hilbert spaces with orthonormal basis, in the sense that given any Hilbert space $\mathcal{H}$ with orthonormal basis $B$, the map $B \ni \vec{x} \mapsto \zeta_{\vec{x}}\in \ell^2(B)$ extends to an isometric linear isomorphism $\mathcal{H} \rightarrow \ell^2(B)$ (\cite{A1}, 3.0.6). 

\begin{theorem}\label{t15}
The following are equivalent.
\begin{itemize}
    \item[\textnormal{i}] The union of a countable collection of finite sets is countable.
    \item[\textnormal{ii}] If $\mathcal{H}$ is a Hilbert space, $B \subseteq \mathcal{H}$ is an orthonormal subset, and $\vec{x}\in \mathcal{H}$, then the set
    \begin{align*}
        T:=\{\vec{e}\in B : \langle \vec{x} , \vec{e} \rangle \neq 0 \}
    \end{align*}
    is countable.
\end{itemize}
\end{theorem}

When AC is at our disposal, ii is easily obtained as a theorem. \ref{t15} makes precise the extent to which the Axiom of Choice is needed to prove this theorem.

\begin{proof}[Proof of \ref{t15}] ($\Rightarrow$) Assume i holds. By \ref{t16}, every probability algebra is ccc. Let $\mathcal{H}$ be a Hilbert space, let $\vec{x}\in \mathcal{H}$, and let $B \subseteq \mathcal{H}$ be an orthonormal subset. Let
\begin{align*}
    N = \sup_{F \subseteq B \text{ finite}} \sum_{\vec{e}\in F} |\langle \vec{x}, \vec{e} \rangle|^2.
\end{align*}
Note that $N \leq ||\vec{x}||^2 < \infty$ by Bessel's inequality. If $N =0$, then clearly $T$ is empty and we are done. So assume $N \neq 0$. This allows us to turn the powerset of $T$ into a probability algebra by defining
\begin{align*}
    \mu(S) = \frac{1}{N} \sup_{F \subseteq S \text{ finite}} \sum_{\vec{e} \in F} |\langle \vec{x}, \vec{e} \rangle|^2
\end{align*}
for all $S \subseteq T$. By \ref{t16}, we conclude that the collection of singletons $\{\{\vec{e}\} : \vec{e} \in T\}$ is countable, as desired.

($\Leftarrow$) For this direction, we employ the framework of \ref{z1} in the case $\kappa = \aleph_0$. Given a triple $(\mathcal{H}, \vec{x}, B)$, where $\mathcal{H}$ is a Hilbert space, $\vec{x} \in \mathcal{H}$ is a vector, and $B\subseteq \mathcal{H}$ is an orthonormal subset, define 
\[
o(\mathcal{H}, \vec{x}, B) = \text{PowerSet}\{\vec{e} \in B : \langle \vec{x} , \vec{e} \rangle \neq 0\},
\] 
ordered by set inclusion. We take $\mathfrak{B}$ to be the proper class of all posets $o(\mathcal{H}, \vec{x}, B)$ (where $(\mathcal{H}, \vec{x}, B)$ ranges over all triples). Note that Assertion ii of \ref{t15} is equivalent to the assertion that every member of~$\mathfrak{B}$ is ccc. Next, we must construct a suitable map $\mathcal{F}: \mathfrak{A}_{\aleph_0} \rightarrow \mathfrak{B}$. Given $Y=(Y_1, Y_2, ...)\in \mathfrak{A}_{\aleph_0}$, we define $\mathcal{F}(Y) = o(\ell^2(\bigcup_n Y_n), \vec{x}_Y, B_Y)$, where $\vec{x}_Y \in \ell^2(\bigcup_n Y_n)$ is defined by $Y_n \ni a \mapsto |Y_n|^{-1}n^{-1} =\vec{x}_Y (a) \in \C$, and $B_Y$ is the standard orthonormal basis for $\ell^2(\bigcup_n Y_n)$. To see that the hypothesis of Theorem~\ref{z1} is satisfied, define the map $f: \bigcup_n Y_n \rightarrow \mathcal{F}(X)$ by $Y_n \ni a \mapsto \{\zeta_a\} = f(a) \in \mathcal{F}(X)$.
\end{proof}

Example 6.3.1 of \cite{A1} shows that the failure of i gives the existence of a Hilbert space with two orthonormal bases of different cardinalities.  It is not clear whether it also implies the existence of a Hilbert space with no orthonormal basis (6.1.2 of \cite{A2}; see however Theorem 6.2.3 of \cite{A1}, where such a space has been constructed from a particularly strong failure of i).  

\section{Conclusion} 
We see that chain conditions are closely related to the cardinality of countable unions. The ZF+AC proof that Cauchy complete metric spaces are $\sigma$-complete makes straightforward, seemingly indispensable use of the Axiom of Countable Choice ($\text{AC}_\omega$). It is well-known that the equivalence of sequential and topological continuity in the context of metric spaces implies $\text{AC}_\omega$. This suggests that theorems about sequential convergence in metric spaces are a good place to look for equivalents of $\text{AC}_\omega$. Moreover, the equivalence of the Baire Category Theorem on Cauchy complete metric spaces (abbreviated BCT) and the Axiom of Dependent Choice demonstrates that theorems about Cauchy complete metric spaces may imply rather strong forms of choice. We are lead to the following.

\begin{question} Can we use the equivalence of Cauchy completeness and $\sigma$-completeness in the context of metric spaces to derive the Axiom of Countable Choice?
\end{question}

Conversely, a typical ZF+AC proof that ccc metric spaces are second-countable (i.e. admit a countable base) uses Zorn's Lemma for posets in which every well-ordered chain is countable. This weak form of Zorn's Lemma (which we call $Z_{\aleph_1}$) turns out to be equivalent to a strong form of Dependent Choice, which, in turn, implies BCT. Given the vastness of the class of Cauchy complete metric spaces, it seems unlikely that a theorem about ccc metric spaces could be used to prove a theorem about Cauchy complete metric spaces. After all, unlike ccc metric spaces, Cauchy complete metric spaces may admit any number of pairwise disjoint non-empty open sets. Ergo, it is improbable that the arrow may be reversed in this case: using the fact that ccc metric spaces are second-countable to prove $Z_{\aleph_1}$. In the interest of viewing this property of ccc metric spaces as a choice principle, we are lead to the following. 

\begin{question} Can the equivalence of the countable chain condition and second-countability in the context of metric spaces be proven using a weaker form of choice than $Z_{\aleph_1}$?
	\end{question}

In \cite{A1} it is asked whether every Hilbert space is a subspace of a Hilbert space admitting an orthonormal basis (in the absence of any form of choice). A good place to start is by looking at spaces of the form $L^2(X, \mathcal{A}, \mu)$ (for measure spaces $(X,\mathcal{A},\mu)$), as this is an abundant source of everyday Hilbert spaces. To simplify matters even further, take $(X,\mathcal{A},\mu)$ to be a measure algebra (i.e. $\mathcal{A}$ is a Boolean algebra and $\mu$ is a finitely additive measure).

\begin{question} Is every space of the form $L^2(X,\mathcal{A},\mu)$ (for $(X,\mathcal{A},\mu)$ a measure algebra) a subspace of a Hilbert space admitting an orthonormal base?
\end{question}

We conclude with a question about the nature of the  specific failure of choice used in this paper. Recall that a set $X$ is Dedekind-finite if there is no injection from $\N$ into $X$. If there is a  Dedekind-finite set $X$ such that the family of all finite subsets of $X$ is not Dedekind-finite, then a simple argument shows that $X$ contains the union of an infinite family of disjoint, nonempty, finite sets. The union of this family is infinite and Dedekind-finite, and therefore not countable. 

\begin{question}
    Assume that the union of any countable collection of finite sets is countable. Can one prove that there is a Dedekind-finite set $X$ such that the set of all finite subsets of $X$ is not Dedekind-finite?
\end{question}



\end{document}